\documentclass{amsart}

\newtheorem{theorem}{Theorem}[section]
\newtheorem{lemma}[theorem]{Lemma}
\newtheorem{claim}[theorem]{Claim}

\usepackage{graphicx,amssymb}

\theoremstyle{definition}

\theoremstyle{remark}
\newtheorem{remark}[theorem]{Remark}

\numberwithin{equation}{section}

\begin{document}

\title[Left-orderability]{Left-orderable fundamental group and Dehn surgery on the knot $5_2$}

\author{Ryoto Hakamata}
\address{Graduate School of Education, Hiroshima University,
1-1-1 Kagamiyama, Higashi-hiroshima, Japan 739-8524.}

\author{Masakazu Teragaito}
\address{Department of Mathematics and Mathematics Education, Hiroshima University,
1-1-1 Kagamiyama, Higashi-hiroshima, Japan 739-8524.}
\email{teragai@hiroshima-u.ac.jp}
\thanks{The second author is partially supported by Japan Society for the Promotion of Science,
Grant-in-Aid for Scientific Research (C), 22540088.
}%





\begin{abstract}
We show that the resulting manifold by $r$-surgery on the knot $5_2$, which is
the two-bridge knot corresponding to the rational number $3/7$, has left-orderable
fundamental group if the slope $r$ satisfies $0\le r \le 4$.
\end{abstract}

\maketitle

\section{Introduction}

A group $G$ is said to be \textit{left-orderable} if
it admits a strict total ordering, which is left invariant.
More precisely, this means that
if $g<h$ then $fg<fh$ for any $f,g,h\in G$.
The fundamental groups of many $3$-manifolds are known to be left-orderable.
On the other hand, the fundamental groups of lens spaces are not
left-orderable, because any left-orderable group is torsion-free.
The notion of an $L$-space is introduced by Ozsv\'{a}th and Szab\'{o} \cite{OS}
in terms of Heegaard-Floer homology.
Lens spaces, 
Seifert fibered manifolds with finite fundamental groups are
typical examples of $L$-spaces.
Although it is an open problem to give a topological characterization of an $L$-space,
there is a conjectured connection between $L$-spaces and left-orderability.
More precisely,
Boyer, Gordon and Watson \cite{BGW} conjecture
that an irreducible rational homology sphere is an $L$-space if and only if
its fundamental group is not left-orderable.
They give affirmative answers for several classes of $3$-manifolds.

It is well known that all knot groups are left-orderable (see \cite{BRW}), but
the resulting closed $3$-manifold by Dehn surgery on a knot does not necessarily have
a left-orderable fundamental group.
For examples, there are many knots which admit Dehn surgery yielding lens spaces. 
By \cite{OS}, the figure-eight knot has no Dehn surgery yielding $L$-spaces.
Hence we can expect that any non-trivial surgery on the figure-eight knot
yields a manifold whose fundamental group is left-orderable, if we support
the conjecture above.
In fact,
Boyer, Gordon and Watson \cite{BGW} show that if $-4<r<4$, then
$r$-surgery on the figure-eight knot yields a manifold whose fundamental group is left-orderable.
In addition, Clay, Lidman and Watson \cite{CLW} verified it for $r=\pm 4$ through 
a different argument.

In this paper, we follow the argument of \cite{BGW} for the most part to handle the knot $5_2$ in the knot table (see \cite{R}).
This knot is the two-bridge knot corresponding to the rational number $3/7$, which is 
a twist knot.
We believe that this is an appropriate target next to the figure-eight knot.
Since $5_2$ is non-fibered, it does not admit Dehn surgery yielding an $L$-space \cite{N}.
Hence we can expect again that any non-trivial Dehn surgery on $5_2$ yields
a $3$-manifold whose fundamental group is left-orderable.

\begin{theorem}\label{thm:main}
Let $K$ be the knot $5_2$. 
If $0\le r\le 4$, then $r$-surgery on $K$ yields a manifold whose fundamental group is left-orderable.
\end{theorem}

In fact, $0$-surgery on any knot yields a prime manifold whose first betti number is $1$, and 
such manifold has left-orderable fundamental group \cite[Corollary 3.4]{BRW}.
Furthermore, the same conclusion holds for $4$-surgery on twist knots \cite{T}.
Hence we will handle the case where $0<r<4$ in this paper.

\section{Knot group and representations}\label{sec:knotgroup}

Let $K$ be the knot $5_2$ in the knot table (\cite{R}).
See Figure \ref{fig:knot}.
This knot is the two-bridge knot corresponding to the rational number $3/7$.
In this diagram, $K$ bounds a once-punctured Klein bottle, as seen from the checkerboard coloring,
whose boundary slope is $4$.
In fact, $4$-surgery on $K$ gives a toroidal manifold, and
$1$, $2$ and $3$-surgeries give small Seifert fibered manifolds (\cite{BW}).

\begin{figure}[ht]
\includegraphics*[scale=0.6]{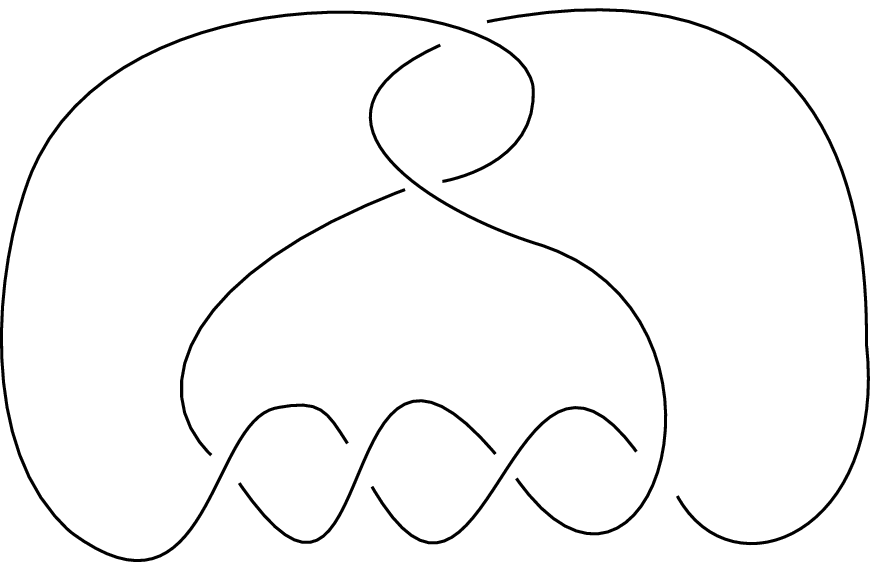}
\caption{}\label{fig:knot}
\end{figure}

Let $M$ be the knot exterior of $K$.
It is well known that the knot group $G=\pi_1(M)$ has a presentation
$\langle x, y\mid wx=yw \rangle$, where $x$ and $y$ are meridians and $w=xyx^{-1}y^{-1}xy$.
Also, a (preferred) longitude $\lambda$ is given by
$x^{-4}w^*w$, where $w^*=yxy^{-1}x^{-1}yx$ corresponds to the reverse word of $w$.
(These facts are easily obtained from Schubert's normal form of the knot \cite{S}.) 

Let $s>0$ be a real number, and let $T=\frac{2+3s+2s^2+\sqrt{s^2+4}}{2s}$.
Then it is easy to see that $T>4$.
Also, let $t=\frac{T+\sqrt{T^2-4}}{2}$. Then, $t>3$ and 
\begin{equation}\label{eq:t}
t=\frac{2+3s+2s^2+\sqrt{s^2+4}+\sqrt{(2+3s+2s^2+\sqrt{s^2+4})^2-16s^2}}{4s}.
\end{equation}

Let $\phi=s(t+t^{-1})^2-(2s^2+3s+2)(t+t^{-1})+s^3+3s^2+4s+3$.
Since $t+t^{-1}=T$, $\phi=sT^2-(2s^2+3s+2)T+s^3+3s^2+4s+3$.
If we solve the equation $\phi=0$ with respect to $T$, we obtain the expression of $T$
in terms of $s$ as above.
Thus $\phi=0$ holds.

We now examine some limits, which will be necessary later.

\begin{lemma}\label{lem:keylimit}
\begin{itemize}
\item[(1)] $\lim_{s\to +0}t=\infty$,
\item[(2)] $\lim_{s\to +0}st=2$,
\item[(3)] $t-s>2$ and $\lim_{s\to\infty}(t-s)=2$,
\item[(4)] $\lim_{s\to \infty}s/t=1$,
\item[(5)] $\lim_{s\to \infty}s(t-s-2)=0$, 
\item[(6)] $\lim_{s\to \infty}t(t-s-2)=0$.
\end{itemize}
\end{lemma}

\begin{proof}
(1) and (2) are obvious from (\ref{eq:t}).
For (3), 
\[
t-s=\frac{2+3s+\sqrt{s^2+4}+\left(\sqrt{(2+3s+2s^2+\sqrt{s^2+4})^2-16s^2}-2s^2\right)}{4s}
\] 
shows $t-s>0$, since $(2+3s+2s^2+\sqrt{s^2+4})^2-16s^2>4s^4$.
The second conclusion follows from
\[
\lim_{s\to\infty}\frac{2+3s+\sqrt{s^2+4}}{4s}=1,\quad \lim_{s\to\infty}\frac{\sqrt{(2+3s+2s^2+\sqrt{s^2+4})^2-16s^2}-2s^2}{4s}=1.
\]

A direct calculation shows (4).

For (5),
\begin{equation*}
\begin{split}
4s(t-s-2)-2&=\left(\sqrt{(2+3s+2s^2+\sqrt{s^2+4})^2-16s^2}+\sqrt{s^2+4}\right)\\
&\quad -(2s^2+5s).
\end{split}
\end{equation*}
Since the right hand side converges to $-2$, we have $\lim_{s\to\infty}s(t-s-2)=0$.

From (3), an inequality $s+2<t<s+3$ holds for sufficiently large $s$.
Then $(s+2)(t-s-2)<t(t-s-2)<(s+3)(t-s-2)$.
Hence (3) and (5) imply (6).
\end{proof}

Let $\rho_s: G\to SL_2(\mathbb{R})$ be the representation defined by the correspondence
\begin{equation}\label{eq:rho}
\rho_s(x)=
\begin{pmatrix}
\sqrt{t} & 0 \\
0 & \frac{1}{\sqrt{t}}
\end{pmatrix},\quad
\rho_s(y)=
\begin{pmatrix}
\frac{t-s-1}{\sqrt{t}-\frac{1}{\sqrt{t}}} & \frac{s}{(\sqrt{t}-\frac{1}{\sqrt{t}})^2}-1 \\
-s & \frac{s+1-\frac{1}{t}}{\sqrt{t}-\frac{1}{\sqrt{t}}}
\end{pmatrix}.
\end{equation}
Here, we remain using the variable $t$ to reduce the complexity.
By using the fact that $s$ and $t$ satisfies the equation $\phi=0$, 
we can check $\rho_s(wx)=\rho_s(yw)$ by a direct calculation.
Hence the correspondence on $x$ and $y$ above gives a homomorphism from $G$ to $SL_2(\mathbb{R})$.
In addition, 
$\rho_s(xy)\ne\rho_s(yx)$, and so $\rho_s$ has the non-abelian image.

\begin{remark}
This representation of $G$ comes from that in \cite[p.786]{Kh}.
The polynomial $\phi$ corresponds to the Riley polynomial \cite{Ri}.
\end{remark}

\begin{lemma}\label{lem:longitude}
For a longitude $\lambda$, $\rho_s(\lambda)$ is diagonal, and
its $(1,1)$-entry is a positive real number.
\end{lemma}

\begin{proof}
Note that $\rho_s(x)$ is diagonal and $\rho_s(x)\ne \pm I$.
The fact that $\rho_s(x)$ commutes with $\rho_s(\lambda)$ easily implies that
$\rho_s(\lambda)$ is also diagonal.
(This can also be seen from a direct calculation of $\rho_s(\lambda)$, by using
$\phi(s,t)=0$.)

A direct calculation gives the $(1,1)$-entry
\begin{equation}\label{eq:B}
\begin{split}
\frac{1}{(t-1)^2 t^5}&
\Bigl(s\left(1 - (2 + s) t + t^2\right) \bigl(s-(2+2s+s^2)t+(1+s)t^2\bigr)^2 \\
& \quad + (1 + s - 
    t)^2 t^3\bigl(s-(1+s)^2t+st^2\bigr)^2\Bigr)
\end{split}
\end{equation}
of $\rho_s(\lambda)$.
Thus it is enough to show that $1-(2+s)t+t^2>0$.
This is equivalent to the inequality $T>2+s$, which is clear from
$T=\frac{2+3s+2s^2+\sqrt{s^2+4}}{2s}$.
\end{proof}

Let $r=p/q$ be a rational number, and let $M(r)$ denote
the resulting manifold by $r$-filling on the knot exterior $M$ of $K$.
In other words, $M(r)$ is obtained by 
attaching a solid torus $V$ to $M$ along their boundaries so that
the loop $x^p\lambda^q$ bounds a meridian disk of $V$.

Clearly, $\rho_s: G\to SL_2(\mathbb{R})$ induces
a homomorphism $\pi_1(M(r))\to SL_2(\mathbb{R})$
if and only if $\rho_s(x)^p\rho_s(\lambda)^q=I$.
Since both of $\rho_s(x)$ and $\rho_s(\lambda)$ are diagonal,
this is equivalent to the equation
\begin{equation}\label{eq:slope}
A_s^p B_s^q=1,
\end{equation}
where $A_s$ and $B_s$ are the $(1,1)$-entries of $\rho_s(x)$ and $\rho_s(\lambda)$, respectively.
We remark that $A_s=\sqrt{t}$ is a positive real number,
so is $B_s$ by Lemma \ref{lem:longitude}.
The equation (\ref{eq:slope}) is furthermore equivalent to
\begin{equation}
-\frac{\log B_s}{\log A_s}=\frac{p}{q}.
\end{equation}

Let $g:(0,\infty)\to \mathbb{R}$ be a function defined by
\[
g(s)=-\frac{\log B_s}{\log A_s}.
\]

\begin{lemma}\label{lem:g-image}
The image of $g$ contains an open interval $(0,4)$.
\end{lemma}

\begin{proof}
First, we show
\[
\lim_{s\to +0}g(s)=0.
\]
Since $\lim_{s\to +0}\log A_s=\infty$,
it is enough to show that $\lim_{s\to +0}B_s=1$.
We decompose $B_s$, given in (\ref{eq:B}), as
\begin{equation}
\begin{split}
B_s &=\frac{s}{t-1}\frac{1-(2+s)t+t^2}{(t-1)t}\left(\frac{s-(2+2s+s^2)t+(1+s)t^2}{t^2}\right)^2
\\
&\quad+\left(\frac{1+s-t}{t-1}\right)^2\left(\frac{s-(1+s)^2t+st^2}{t}\right)^2.
\end{split}
\end{equation}
From Lemma \ref{lem:keylimit},
$\lim_{s\to +0}t=\infty$ and $\lim_{s\to +0}st=2$.
These give
\[
\lim_{s\to +0}\frac{s}{t-1}=0, \quad \lim_{s\to +0}\frac{1-(2+s)t+t^2}{(t-1)t}=1,
\]
\[
\lim_{s\to +0}\frac{s-(2+2s+s^2)t+(1+s)t^2}{t^2}=1, \lim_{s\to +0}\frac{1+s-t}{t-1}=-1,
\]
and
\[
\lim_{s\to +0}\frac{s-(1+s)^2t+st^2}{t}=1.
\]
Thus we have $\lim_{s\to +0}B_s=0$.

Second, we show
\[
\lim_{s\to\infty}g(s)=4.
\]
Let $N$ be the numerator of $B_s$ shown in (\ref{eq:B}).
Then
\[
\frac{\log B_s}{\log A_s}=\frac{2\log N}{\log t}-\frac{2\log (t-1)^2t^5}{\log t}.
\]

\begin{claim}\label{cl:g}
$\lim_{s\to \infty}Nt^{-5}=1$.
\end{claim}

\begin{proof}[Proof of Claim \ref{cl:g}]
From Lemma \ref{lem:keylimit}, $\lim_{s\to \infty}s/t=1$, and $\lim_{s\to\infty}(1+s-t)=-1$.
We have
\begin{eqnarray*}
1-(2+s)t+t^2&=& t(t-s-2)+1,\\
\frac{s-(1+s)^2t+st^2}{t}&=&\frac{s}{t}+s(t-s-2)-1,\\
\frac{s-(2+2s+s^2)t+(1+s)t^2}{t^2}&=& \frac{1}{t}\cdot\frac{s-(1+s)^2t+st^2}{t}-\frac{1}{t}+1.
\end{eqnarray*}
Hence Lemma \ref{lem:keylimit} implies
\[
\lim_{s\to \infty}(1-(2+s)t+t^2)=
\lim_{s\to\infty}\frac{s-(2+2s+s^2)t+(1+s)t^2}{t^2}=1,
\]
\[
\lim_{s\to\infty}\frac{s-(1+s)^2t+st^2}{t}=0.
\]
Combining these, we have $\lim_{s\to \infty}Nt^{-5}=1$.
\end{proof}


Thus we have
$\lim_{s\to \infty}(\log N-5\log t)=0$.
Then 
\[
\lim_{s\to \infty}\frac{\log N}{\log t}=5.
\]
Clearly, 
\[
\lim_{t\to\infty}\frac{\log (t-1)^2t^5}{\log t}=7.
\]
Hence we have $\lim_{s\to \infty}g(s)=4$.
\end{proof}

\section{The universal covering group of $SL_2(\mathbb{R})$}\label{sec:univ}

Let
\[
SU(1,1)=\left\{
\begin{pmatrix}
\alpha & \beta \\
\bar{\beta} & \bar{\alpha}\\
\end{pmatrix}
\mid |\alpha|^2-|\beta|^2=1\right\}
\]
be the special unitary group over $\mathbb{C}$ of signature $(1,1)$.
It is well known that $SU(1,1)$ is conjugate to $SL_2(\mathbb{R})$ in $GL_2(\mathbb{C})$.
The correspondence is given by
$\psi:SL_2(\mathbb{R})\to SU(1,1)$, sending
$A\mapsto JAJ^{-1}$, where
\[J=
\begin{pmatrix}
1 & -i \\
1 & i
\end{pmatrix}.
\]
Thus
\[
\psi:
\begin{pmatrix}
a & b \\
c & d
\end{pmatrix}
\mapsto 
\begin{pmatrix}
\frac{a+d+(b-c)i}{2} & \frac{a-d-(b+c)i}{2} \\
\frac{a-d+(b+c)i}{2} & \frac{a+d-(b-c)i}{2}
\end{pmatrix}.
\]

There is a parametrization of $SU(1,1)$ by $(\gamma,\omega)$
 where $\gamma=\beta/\alpha$ and $\omega=\arg \alpha$ defined mod $2\pi$ (see \cite{B}).
Thus 
$SU(1,1)=\{(\gamma,\omega) \mid |\gamma|<1, -\pi\le \omega<\pi\}$.
Topologically, $SU(1,1)$ is an open solid torus $\Delta\times S^1$, where
$\Delta=\{\gamma\in \mathbb{C}\mid |\gamma|<1\}$. 
The group operation is given by
$(\gamma,\omega)(\gamma',\omega')=(\gamma'',\omega'')$, where
\begin{eqnarray}\label{eq:sl2R}
\gamma''&=& \frac{\gamma'+\gamma e^{-2i\omega'}}{1+\gamma\bar{\gamma'}e^{-2i\omega'}},\label{eq:sl2R1}\\
\omega''&=& \omega+\omega'+\dfrac{1}{2i}\log
\frac{1+\gamma\bar{\gamma'}e^{-2i\omega'}}{1+\bar{\gamma}\gamma'e^{2i\omega'}}.
\label{eq:sl2R2}
\end{eqnarray}
These equations come from the matrix operation.
Here, the logarithm function is defined by its principal value and $\omega''$ is defined by mod $2\pi$.
The identity element is $(0,0)$, and
the correspondence between
$\begin{pmatrix}
\alpha & \beta \\
\bar{\beta} & \bar{\alpha}\\
\end{pmatrix}$ and $(\gamma,\omega)$ gives an isomorphism.

Now, the universal covering group $\widetilde{SL_2(\mathbb{R})}$ of $SU(1,1)$
can be described as
\[
\widetilde{SL_2(\mathbb{R})}=\{(\gamma,\omega)\mid |\gamma|<1, -\infty<\omega<\infty\}.
\]
Thus $\widetilde{SL_2(\mathbb{R})}$ is homeomorphic to $\Delta\times \mathbb{R}$. 
The group operation is given by (\ref{eq:sl2R1}) and (\ref{eq:sl2R2}) again, but
$\omega''$ is not mod $2\pi$ anymore.

Let $\Phi: \widetilde{SL_2(\mathbb{R})}\to SL_2(\mathbb{R})$ be the covering projection.
Then it is obvious that $\ker\Phi=\{(0,2m\pi)\mid m\in \mathbb{Z}\}$.

\begin{lemma}
The subset $(-1,1)\times \{0\}$ of $\widetilde{SL_2(\mathbb{R})}$ forms a subgroup.
\end{lemma}

\begin{proof}
From (\ref{eq:sl2R1}) and (\ref{eq:sl2R2}),
it is straightforward to see that $(-1,1)\times \{0\}$ is closed under the group operation.
For $(\gamma,0)\in (-1,1)\times\{0\}$, its inverse is $(-\gamma,0)$.
\end{proof}

For the representation $\rho_s : G\to SL_2(\mathbb{R})$ defined by (\ref{eq:rho}),
\begin{equation}
\psi(\rho_s(x))=\frac{1}{2\sqrt{t}}
\begin{pmatrix}
t+1 & t-1 \\
t-1 & t+1
\end{pmatrix}\in SU(1,1).
\end{equation}
Thus $\psi(\rho_s(x))$ corresponds to $(\gamma_x,0)$, where $\gamma_x=\dfrac{t-1}{t+1}$.


Also, for a longitude $\lambda$,
\[
\psi(\rho_s(\lambda))=\frac{1}{2}
\begin{pmatrix}
B_s+\frac{1}{B_s} & B_s-\frac{1}{B_s} \\
B_s-\frac{1}{B_s} & B_s+\frac{1}{B_s}
\end{pmatrix}, B_s>0
\]
from Lemma \ref{lem:longitude}.
Thus $\psi(\rho_s(\lambda))$ corresponds to $(\gamma_\lambda,0)$, where
$\gamma_\lambda=\dfrac{B_s^2-1}{B_s^2+1}$.

\section{Proof of Theorem}

As the knot exterior $M$ satisfies $H^2(M;\mathbb{Z})=0$,
any $\rho_s: G\to SL_2(\mathbb{R})$ lifts to a representation
$\tilde{\rho}: G\to\widetilde{SL_2(\mathbb{R})}$ \cite{G}.
Moreover, any two lifts $\tilde{\rho}$ and $\tilde{\rho}'$ are
related as follows:
\[
\tilde{\rho}'(g)=h(g)\tilde{\rho}(g),
\]
where $h:G\to \ker \Phi\subset\widetilde{SL_2(\mathbb{R})}$.
Since $\ker \Phi=\{(0,2m\pi)\mid m\in \mathbb{Z}\}$ is isomorphic to $\mathbb{Z}$,
the homomorphism $h$ factors through $H_1(M)$, so
it is determined only by the value $h(x)$ of a meridian $x$ (see \cite{Kh}).

The following result is the key in \cite{BGW}, which is originally claimed in \cite{Kh}, for the figure eight knot. 
Our proof most follows that of \cite{BGW}, but it is much simpler,
because of the values of $\psi(\rho_s(x))$ and $\psi(\rho_s(\lambda))$, which
are calculated in Section \ref{sec:univ}.

\begin{lemma}\label{lem:key}
Let $\tilde{\rho}: G\to \widetilde{SL_2(\mathbb{R})}$ be a lift of $\rho_s$.
Then replacing $\tilde{\rho}$ by a representation
$\tilde{\rho}'=h\cdot \tilde{\rho}$ for some $h:G\to \widetilde{SL_2(\mathbb{R})}$,
we can suppose that $\tilde{\rho}(\pi_1(\partial M))$ is contained in the subgroup $(-1,1)\times \{0\}$ of $\widetilde{SL_2(\mathbb{R})}$.
\end{lemma}

\begin{proof}
Since $\Phi(\tilde{\rho}(\lambda))=(\gamma_\lambda,0)$, $\gamma_\lambda\in (-1,1)$,
$\tilde{\rho}(\lambda)=(\gamma_\lambda,2j\pi)$ for some $j$.
On the other hand, $\lambda$ is a commutator,
because our knot is genus one.
Therefore the inequality (5.5) of \cite{W} implies  $-3\pi/2<2j\pi<3\pi/2$.
Thus we have $\tilde{\rho}(\lambda)=(\gamma_\lambda,0)$.

Similarly, $\tilde{\rho}(x)=(\gamma_x,2\ell \pi)$ for some $\ell$, where $\gamma_x\in (-1,1)$.
Let us choose $h:G\to \widetilde{SL_2(\mathbb{R})}$
so that $h(x)=(0,-2\ell\pi)$.
Set $\tilde{\rho}'=h\cdot \tilde{\rho}$.
Then a direct calculation shows that $\tilde{\rho}'(x)=(\gamma_x,0)$ and $\tilde{\rho}'(\lambda)=(\gamma_\lambda,0)$.
Since $x$ and $\lambda$ generate the peripheral subgroup $\pi_1(\partial M)$,
the conclusion follows from these.
\end{proof}

\begin{proof}[Proof of Theorem \ref{thm:main}]
Let $r=p/q\in (0,4)$.
By Lemma \ref{lem:g-image}, we can fix $s$ so that $g(s)=r$.
Choose a lift $\tilde{\rho}$ of $\rho_s$ so that
$\tilde{\rho}(\pi_1(\partial M))\subset (-1,1)\times\{0\}$.
Then $\rho_s(x^p\lambda^q)=I$, so $\Phi(\tilde{\rho}(x^p\lambda^q))=I$.
This means that $\tilde{\rho}(x^p\lambda^q)$ lies in $\ker\Phi=\{(0,2m\pi)\mid m\in \mathbb{Z}\}$.
Hence $\tilde{\rho}(x^p\lambda^q)=(0,0)$.
Then $\tilde{\rho}$ can induce a homomorphism $\pi_1(M(r))\to \widetilde{SL_2(\mathbb{R})}$
with non-abelian image.
Recall that $\widetilde{SL_2(\mathbb{R})}$ is left-orderable \cite{Be}.
Since $M(r)$ is irreducible \cite{HT}, 
$\pi_1(M(r))$ is left-orderable by \cite[Theorem 1.1]{BRW}.
This completes the proof.
\end{proof}

\bibliographystyle{amsplain}

\begin{thebibliography}{BGW}

\bibitem{B}
V. Bargmann, 
\textit{Irreducible unitary representations of the Lorentz group},
Ann. of Math. \textbf{48} (1947), 568--640. 

\bibitem{Be}
G. Bergman, 
\textit{Right orderable groups that are not locally indicable}, 
Pacific J. Math. \textbf{147} (1991), 243--248. 


\bibitem{BGW}
S. Boyer, C. McA. Gordon and L. Watson,
\textit{On $L$-spaces and left-orderable fundamental groups},
preprint, \texttt{arXiv:1107.5016}.

\bibitem{BRW}
S. Boyer, D. Rolfsen and B. Wiest,
\textit{Orderable 3-manifold groups},
Ann. Inst. Fourier (Grenoble) \textbf{55} (2005), 243--288.

\bibitem{BW}
M. Brittenham and  Y. Q. Wu,
\textit{The classification of exceptional Dehn surgeries on 2-bridge knots},
Comm. Anal. Geom. \textbf{9} (2001), 97--113.


\bibitem{CLW}
A. Clay, T. Lidman and L. Watson,
\textit{Graph manifolds, left-orderability and amalgamation},
preprint, \texttt{arXiv:1106.0486}.

\bibitem{CT}
A. Clay and M. Teragaito,
\textit{Left-orderability and exceptional Dehn surgery on two-bridge knots},
to appear in the Proceedings of Geometry and Topology Down Under,
Contemporary Mathematics Series.

\bibitem{G}
E. Ghys, 
\textit{Groups acting on the circle},
Enseign. Math. \textbf{47} (2001), 329--407. 


\bibitem{HT}
A. Hatcher and W. Thurston,
\textit{Incompressible surfaces in 2-bridge knot complements},
Invent. Math. \textbf{79} (1985), 225--246.

\bibitem{Kh}
V. T. Khoi,
\textit{A cut-and-paste method for computing the Seifert volumes},
Math. Ann. \textbf{326} (2003), 759--801. 


\bibitem{N}
Y. Ni,
\textit{Knot Floer homology detects fibred knots},
Invent. Math. \textbf{170} (2007), 577--608.

\bibitem{OS}
P. Ozsv\'{a}th and Z. Szab\'{o},
\textit{On knot Floer homology and lens space surgeries},
Topology \textbf{44} (2005), 1281--1300. 

\bibitem{Ri}
R. Riley,
\textit{Nonabelian representations of 2-bridge knot groups}, 
Quart. J. Math. Oxford Ser. (2) \textbf{35} (1984), 191--208. 


\bibitem{R}
D. Rolfsen,
\textit{Knots and links},
Mathematics Lecture Series, No. 7. Publish or Perish, Inc., Berkeley, Calif., 1976.

\bibitem{S}
H. Schubert,
\textit{Knoten mit zwei Br\"{u}cken},
Math. Z. \textbf{65} (1956), 133--170. 

\bibitem{T}
M. Teragaito,
\textit{Left-orderability and exceptional Dehn surgery on twist knots},
to appear in Canad. Math. Bull.

\bibitem{W}
J. Wood,
\textit{Bundles with totally disconnected structure group}, 
Comment. Math. Helv. \textbf{46} (1971), 257--273. 


\end{thebibliography}

\end{document}